\documentclass[runningheads,a4paper]{llncs}
\usepackage{amssymb,amsmath}
\usepackage{subfig}
\usepackage{graphicx}
\usepackage{url}
\usepackage{tikz}
\usetikzlibrary{snakes,arrows,shapes}

\newtheorem{observation}{Observation}
\newtheorem{conj}{Conjecture}

\newcommand{\alow}{\ensuremath{a_\ell}}
\newcommand{\ahigh}{\ensuremath{a_h}}
\newcommand{\blow}{\ensuremath{b_\ell}}
\newcommand{\bhigh}{\ensuremath{b_h}}

\renewcommand{\L}{\ensuremath{L}}
\newcommand{\T}{\ensuremath{\mathcal{T}}}
\renewcommand{\S}{\ensuremath{\mathcal{S}}}
\newcommand{\keywords}[1]{\par\addvspace\baselineskip\noindent\keywordname\enspace\ignorespaces#1}

\begin{document}
\mainmatter  
\title{Improved Lower Bounds on the\\Compatibility of Multi-State Characters\thanks{This work was supported in part by the National Science Foundation under grants CCF-1017189 and DEB-0829674.}}
\titlerunning{Improved Lower Bounds on the Compatibility of Multi-State Characters}
\author{Brad Shutters \and Sudheer Vakati \and David Fern\'{a}ndez-Baca}
\authorrunning{B. Shutters \and S. Vakati \and D. Fern\'{a}ndez-Baca}
\institute{Department of Computer Science, Iowa State University, Ames, IA\ \ 50011, USA \\\email{\{shutters,svakati,fernande\}@iastate.edu}}
\maketitle
\begin{abstract}
We study a long standing conjecture on the necessary and sufficient conditions for the compatibility of multi-state characters: There exists a function $f(r)$ such that, for any set $C$ of $r$-state characters, $C$ is compatible if and only if every subset of $f(r)$ characters of $C$ is compatible. We show that for every $r \ge 2$, there exists an incompatible set $C$ of $\lfloor\frac{r}{2}\rfloor\cdot\lceil\frac{r}{2}\rceil + 1$ $r$-state characters such that every proper subset of $C$ is compatible. Thus, $f(r) \ge \lfloor\frac{r}{2}\rfloor\cdot\lceil\frac{r}{2}\rceil + 1$ for every $r \ge 2$. This improves the previous lower bound of $f(r) \ge r$ given by Meacham (1983), and generalizes the construction showing that $f(4) \ge 5$ given by Habib and To (2011). We prove our result via a result on quartet compatibility that may be of independent interest: For every integer $n \ge 4$, there exists an incompatible set $Q$ of $\lfloor\frac{n-2}{2}\rfloor\cdot\lceil\frac{n-2}{2}\rceil + 1$ quartets over $n$ labels such that every proper subset of $Q$ is compatible. We contrast this with a result on the compatibility of triplets: For every $n \ge 3$, if $R$ is an incompatible set of more than $n-1$ triplets over $n$ labels, then some proper subset of $R$ is incompatible. We show this upper bound is tight by exhibiting, for every $n \ge 3$, a set of $n-1$ triplets over $n$ taxa such that $R$ is incompatible, but every proper subset of $R$ is compatible.
\keywords{phylogenetics, quartet compatibility, character compatibility, perfect phylogeny}
\end{abstract}
\section{Introduction}

The multi-state character compatibility (or perfect phylogeny) problem is a basic question in computational phylogenetics \cite{Semple2003a}. Given a set  $C$ of characters, we are asked whether there exists a phylogenetic tree that displays every character in $C$; if so, $C$ is said to be compatible, and incompatible otherwise. The problem is known to be NP-complete \cite{Bodlaender1992a,Steel1992a}, but certain special cases are known to be polynomially-solvable \cite{Agarwala1994a,Dress1992a,Gusfield1991a,Kannan1994a,Kannan1997a,Lam2011a,Shutters2012a}. See \cite{FernandezBaca2001a} for more on the perfect phylogeny problem.
 
In this paper we study a long standing conjecture on the necessary and sufficient conditions for the compatibility of multi-state characters.

\begin{conj}\label{conjFR}
There exists a function $f(r)$ such that, for any set $C$ of $r$-state characters, $C$ is compatible if and only if every subset of $f(r)$ characters of $C$ is compatible.
\end{conj}

If Conjecture \ref{conjFR} is true, it would follow that we can determine if any set $C$ of $r$-state characters is compatible by testing the compatibility of each subset of $f(r)$ characters of $C$, and, in case of incompatibility, output a subset of at most $f(r)$ characters of $C$ that is incompatible.

A classic result on binary character compatibility shows that $f(2)=2$; see \cite{Buneman1971a,Estabrook1976a,Gusfield1991a,Meacham1983a,Semple2003a}. 
In 1975, Fitch \cite{Fitch1975a,Fitch1977a} gave an example of a set $C$ of three 3-state characters such that $C$ is incompatible, but every pair of characters in $C$ is compatible, showing that $f(3) \ge 3$. In 1983, Meacham \cite{Meacham1983a} generalized this example to $r$-state characters for every $r \ge 3$, demonstrating a lower bound of $f(r) \ge r$ for all $r$; see also \cite{Lam2011a}.
A recent breakthrough by Lam, Gusfield, and Sridhar \cite{Lam2011a} showed that $f(3)=3$.  While the previous results could lead one to conjecture that $f(r)=r$ for all $r$, Habib and To \cite{Habib2011a} recently disproved this possibility by exhibiting a set $C$ of five 4-state characters such that $C$ is incompatible, but every proper subset of the characters in $C$ are compatible, showing that $f(4) \ge 5$. They conjectured that $f(r) \ge r+1$ for every $r \ge 4$.

The main result of this paper is to prove the conjecture stated in \cite{Habib2011a} by giving a quadratic lower bound on $f(r)$.  Formally, we show that for every integer $r \ge 2$, there exists a set $C$ of $r$-state characters such that all of the following conditions hold.
\begin{enumerate}
	\item $C$ is incompatible.
	\item Every proper subset of $C$ is compatible.
	\item $|C| = \lfloor\frac{r}{2}\rfloor\cdot\lceil\frac{r}{2}\rceil + 1$.
\end{enumerate}
Therefore, $f(r) \ge  \lfloor\frac{r}{2}\rfloor\cdot\lceil\frac{r}{2}\rceil + 1$ for every $r \ge 2$.

Our proof relies on a new result on quartet compatibility which we believe is of independent interest. We show that for every integer $n \ge 4$, there exists a set $Q$ of quartets over a set of $n$ labels such that all of the following conditions hold.
\begin{enumerate}
	\item $Q$ is incompatible.
	\item Every proper subset of $Q$ is compatible.
	\item $|Q|=\lfloor\frac{n-2}{2}\rfloor\cdot\lceil\frac{n-2}{2}\rceil + 1$.
\end{enumerate}
This represents an improvement over the previous lower bound on the maximum cardinality of such an incompatible set of quartets of  $n-2$ given in \cite{Steel1992a}.

We contrast our result on quartet compatibility with a result on the compatibility of triplets: For every $n \ge 3$, if $R$ is an incompatible set of triplets over $n$ labels, and $|R| > n-1$, then some proper subset of $R$ is incompatible. We show this upper bound is tight by exhibiting, for every $n \ge 3$, a set of $n-1$ triplets over $n$ labels such that $R$ is incompatible, but every proper subset of $R$ is compatible. The results given here on the compatibility of triplets appear to have been previously known \cite{SteelCommunications}, but are formally proven here.

\section{Preliminaries}

Given a graph $G$, we represent the vertices and edges of $G$ by $V(G)$ and $E(G)$ respectively.  We use the abbreviated notation $uv$ for an edge $\{u,v\} \in E(G)$. For any $e \in E(G)$, $G-e$ represents the graph obtained from $G$ by deleting edge $e$. For any integer $i$, we use $[i]$ to represent the set $\{1, 2, \cdots, i\}$.

\subsection{Unrooted Phylogenetic Trees}

An {\em unrooted phylogenetic tree} (or just {\em tree}) is a tree $T$ whose leaves are in one to one correspondence with a label set $\L(T)$, and has no vertex of degree two. See Fig. \ref{figCompatibleTree}(a) for an example. For a collection $\T$ of trees, the {\em label set} of $\T$, denoted $\L(\T)$, is the union of the label sets of the trees in $\T$. 
A tree is {\em binary} if every internal (non-leaf) vertex has degree three. A {\em quartet} is a binary tree with exactly four leaves. A quartet with label set $\{a,b,c,d\}$ is denoted $ab|cd$ if the path between the leaves labeled $a$ and $b$ does not intersect with the path between the leaves labeled $c$ and $d$. 
 

\begin{figure}[!b]
\centering
\subfloat[]{
\begin{tikzpicture}[scale=1]
\node (mid) at (0:0) [circle,inner sep=1.5,fill=black] {};
\node (v1) at (120:.6) [circle,inner sep=1.5,fill=black] {};
\node (v2) at (240:.6) [circle,inner sep=1.5,fill=black] {};
\node (v3) at (0:.6) [circle,inner sep=1.5,fill=black] {};
\node (a) at (90:1.2) [circle,inner sep=1.5,fill=black,label=right:$a$] {};
\node (b) at (150:1.2) [circle,inner sep=1.5,fill=black,label=left:$b$] {};
\node (c) at (210:1.2) [circle,inner sep=1.5,fill=black,label=left:$c$] {};
\node (d) at (270:1.2) [circle,inner sep=1.5,fill=black,label=right:$d$] {};
\node (e) at (30:1.2) [circle,inner sep=1.5,fill=black,label=right:$e$] {};
\node (f) at (330:1.2) [circle,inner sep=1.5,fill=black,label=right:$f$] {};
\path[-] (mid) edge (v1);
\path[-] (mid) edge (v2);
\path[-] (mid) edge (v3);
\path[-] (v1) edge (a);
\path[-] (v1) edge (b);
\path[-] (v2) edge (c);
\path[-] (v2) edge (d);
\path[-] (v3) edge (e);
\path[-] (v3) edge (f);
\end{tikzpicture}
}
\hspace{3em}
\qquad
\hspace{3em}
\subfloat[]{
\begin{tikzpicture}[scale=1]
\node (mid) at (0:0) [circle,inner sep=1.5,fill=black] {};
\node (v1) at (180:.6) [circle,inner sep=1.5,fill=black] {};
\node (v2) at (0:.6) [circle,inner sep=1.5,fill=black] {};
\node (a) at (150:1.2) [circle,inner sep=1.5,fill=black,label=right:$a$] {};
\node (b) at (210:1.2) [circle,inner sep=1.5,fill=black,label=left:$b$] {};
\node (c) at (30:1.2) [circle,inner sep=1.5,fill=black,label=left:$c$] {};
\node (d) at (330:1.2) [circle,inner sep=1.5,fill=black,label=right:$d$] {};
\node (e) at (90:.6) [circle,inner sep=1.5,fill=black,label=right:$e$] {};
\path[-] (mid) edge (v1);
\path[-] (mid) edge (v2);
\path[-] (mid) edge (e);
\path[-] (v1) edge (a);
\path[-] (v1) edge (b);
\path[-] (v2) edge (c);
\path[-] (v2) edge (d);
\end{tikzpicture}
}
\caption{(a) shows a tree $T$ witnessing that the quartets $q_1=ab|ce$, $q_2=cd|bf$, and $q_3=ad|ef$ are compatible; $T$ is also a witness that the characters $\chi_{q_1} = ab|ce|d|f$, $\chi_{q_2}=cd|bf|a|e$, and $\chi_{q_3}=ad|ef|b|c$ are compatible; (b) shows $T|\{a,b,c,d,e\}$.}
\label{figCompatibleTree}
\end{figure}
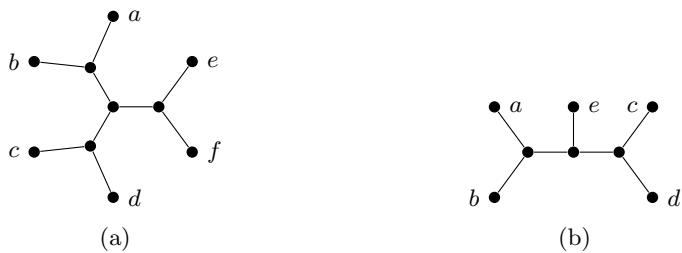

For a tree $T$, and a label set $L \subseteq \L(T)$, the {\em restriction} of $T$ to $L$, denoted by $T|L$, is the tree obtained from the minimal subtree of $T$ connecting all the leaves with labels in $L$ by suppressing vertices of degree two. 
See Fig. \ref{figCompatibleTree}(b) for an example. A tree $T$ {\em displays} another tree $T'$, if $T'$ can be obtained from $T|\L(T')$ by contracting edges. A tree $T$ displays a collection of trees $\T$ if $T$ displays every tree in $\T$. If such a tree $T$ exists, then we say that $\T$ is {\em compatible}; otherwise, we say that $\T$ is {\em incompatible}. See Fig. \ref{figCompatibleTree}(a) for an example. Determining if a collection of unrooted trees is compatible is NP-complete \cite{Steel1992a}. 


\subsection{Multi-State Characters}

There is also a notion of compatibility for sets of partitions of a label set $L$.
A {\em character} $\chi$ on $L$ is a partition of $L$; the parts of $\chi$ are called {\em states}. If $\chi$ has at most $r$ parts, then $\chi$ is an $r$-state character. Given a tree $T$ with $L=L(T)$ and a state $s$ of $\chi$, we denote by $T_s(\chi)$ the minimal subtree of $T$ connecting all leaves with labels having state $s$ for $\chi$. We say that $\chi$ is {\em convex} on $T$, or equivalently $T$ {\em displays} $\chi$, if the subtrees $T_i(\chi)$ and $T_j(\chi)$ are vertex disjoint for all states $i$ and $j$ of $\chi$ where $i \ne j$. 
A collection $C$ of characters is {\em compatible} if there exists a tree $T$ on which every character in $C$ is convex. If no such tree exists, then we say that $C$ is {\em incompatible}.
See Fig. \ref{figCompatibleTree}(a) for an example. The {\em perfect phylogeny problem} (or {\em character compatibility problem}) is to determine whether a given set of characters is compatible.

There is a natural correspondence between quartet compatibility and character compatibility that we now describe. Let $Q$ be a set of quartets, $n=|\L(Q)|$, and $r=n-2$. For each $q=ab|cd \in Q$, we define the $r$-state character corresponding to $q$, denoted $\chi_q$, as the character where $a$ and $b$ have state 0 for $\chi_q$; $c$ and $d$ have state 1 for $\chi_q$; and, for each $\ell \in \L(Q) \setminus \{a,b,c,d\}$, there is a state $s$ of $\chi_q$ such that $\ell$ is the only label with state $s$ for character $\chi_q$ (see Example \ref{exCcorQ}). We define the set of $r$-state characters corresponding to $Q$ by
$C_Q = \bigcup_{q \in Q} \{\chi_q\}.$

\begin{example}\label{exCcorQ}
Consider the quartets and characters given in Fig. \ref{figCompatibleTree}(a): $\chi_{q_1}$ is the character corresponding to $q_1$, $\chi_{q_2}$ is the character corresponding to $q_2$, and $\chi_{q_3}$ is the character corresponding to $q_3$.
\end{example}

The proof of the following lemma relating quartet compatibility to character compatibility is straightforward, and given in the appendix.

\begin{lemma}\label{lemQtoC}
A set $Q$ of quartets is compatible if and only if $C_Q$ is compatible.
\end{lemma}


\subsection{Quartet Graphs}


\begin{figure}[!b]
\centering
\subfloat[$G_Q$]{
	\centering
	\begin{tikzpicture}[scale=1]
	\node (f) at (0:1) [circle,inner sep=1.5,fill=black,label=right:$f$] {};
	\node (b) at (60:1) [circle,inner sep=1.5,fill=black,label=above:$b$] {};
	\node (a) at (120:1) [circle,inner sep=1.5,fill=black,label=above:$a$] {};
	\node (d) at (180:1) [circle,inner sep=1.5,fill=black,label=left:$d$] {};
	\node (c) at (240:1) [circle,inner sep=1.5,fill=black,label=below:$c$] {};
	\node (e) at (300:1) [circle,inner sep=1.5,fill=black,label=below:$e$] {};
	
	\path[-] (a) edge node[above, inner sep=1]{$q_{1}$} (b);
	\path[-] (c) edge node[below, inner sep=1]{$q_{1}$} (e);
	\path[-] (c) edge node[below left, inner sep=1]{$q_{2}$} (d);
	\path[-] (b) edge node[above right, inner sep=1]{$q_{2}$} (f);
	\path[-] (e) edge node[below right, inner sep=1]{$q_{3}$} (f);
	\path[-] (a) edge node[above left, inner sep=1]{$q_{3}$} (d);
	\end{tikzpicture}
}
\quad
\subfloat[$\{a,b\} \rightarrow g$]
{
	\centering
	\begin{tikzpicture}[scale=1]
	\node (f) at (0:1) [circle,inner sep=1.5,fill=black,label=right:$f$] {};
	\node (ab) at (90:.75) [circle,inner sep=1.5,fill=black,label=above:$g$] {};
	\node (d) at (180:1) [circle,inner sep=1.5,fill=black,label=left:$d$] {};
	\node (c) at (240:1) [circle,inner sep=1.5,fill=black,label=below:$c$] {};
	\node (e) at (300:1) [circle,inner sep=1.5,fill=black,label=below:$e$] {};
	
	\path[-] (c) edge node[below left, inner sep=1]{$q_{2}$} (d);
	\path[-] (ab) edge node[above right, inner sep=1]{$q_{2}$} (f);
	\path[-] (e) edge node[below right, inner sep=1]{$q_{3}$} (f);
	\path[-] (ab) edge node[above left, inner sep=1]{$q_{3}$} (d);
	\end{tikzpicture}
}
\quad
\subfloat[$\{e,f\} \rightarrow h$]
{
	\centering
	\begin{tikzpicture}[scale=1]
	\node (ef) at (330:.75) [circle,inner sep=1.5,fill=black,label=below:$h$] {};
	\node (ab) at (90:.75) [circle,inner sep=1.5,fill=black,label=above:$g$] {};
	\node (d) at (180:.75) [circle,inner sep=1.5,fill=black,label=above:$d$] {};
	\node (c) at (240:.75) [circle,inner sep=1.5,fill=black,label=below:$c$] {};
	
	\path[-] (c) edge node[above right, inner sep=1]{$q_{2}$} (d);
	\path[-] (ab) edge node[above right, inner sep=1]{$q_{2}$} (ef);
	\end{tikzpicture}
}
\quad
\subfloat[$\{g,h\} \rightarrow i$]
{
	\centering
	\hspace{5pt}
	\begin{tikzpicture}[scale=1]
	\node (abef) at (45:.5) [circle,inner sep=1.5,fill=black,label=above:$i$] {};
	\node (d) at (180:.5) [circle,inner sep=1.5,fill=black,label=above:$d$] {};
	\node (c) at (270:.5) [circle,inner sep=1.5,fill=black,label=below:$c$] {};
	\end{tikzpicture}
	\hspace{5pt}
}
\caption{The quartet graph $G_Q$ for the set of quartets given in Fig. \ref{figCompatibleTree} is shown in (a). A complete unification sequence for $G_Q$ is shown in (a)--(d). In the figures, the unification of vertices $U$ to a new vertex $u$ is denoted $U \rightarrow u$.}
\label{figQuartetGraph}
\end{figure}
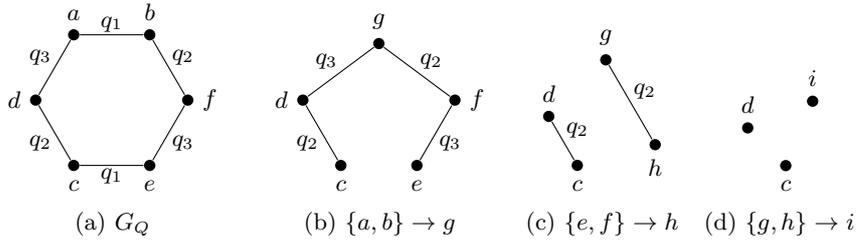

We now give a brief overview of quartet graphs which were introduced in \cite{Grunewald2008a}, and characterize when a collection of quartets is compatible.

Let $Q$ be a collection of quartets with label set $L$.  The {\em quartet graph} $G_{Q}$ on $Q$ is the edge-colored graph defined as follows. 

\begin{enumerate}
\item There is a vertex $\ell$ in $G_Q$ for each $\ell$ in $L$. 

\item For every quartet $q = ab | cd$  in $Q$, $G_{Q}$ has two edges $ab$ and $cd$, both labeled by $q$.
We call an edge labeled by $q$ a {\em $q$-colored} edge.
\end{enumerate}

An example of a quartet graph is given in Fig.~\ref{figQuartetGraph}(a). Let $G$ be any edge colored graph. Let $U \subseteq V(G)$ such that for any color $c$, at most one $c$-colored edge is incident on the vertices of $U$. The {\em unification} of vertices $U$ in $G$ is the graph $G'$ obtained from $G$ as follows:

\begin{enumerate}
\item Add a new vertex $u$ to $G'$. 
\item For each $c$-colored edge $vw$ in $G$ where $v \in U$ and $w \not\in U$, add a $c$-colored edge $uw$ to $G'$.
\item For each $c$-colored edge $vw$ in $G$ where both $v \in U$ and $w \in U$, delete all $c$-colored edges from $G'$.
\item Delete all vertices in $U$, and edges incident to vertices in $U$, from $G'$.
\end{enumerate}

See Figures \ref{figQuartetGraph} and \ref{figQ43QuartetGraph} for examples of unification.
A {\em unification sequence} for $G$ is a sequence $G_0\!\!=\!\!G, G_1, G_2, \ldots, G_k$ of graphs where, for any $i > 0$, $G_i$ is derived from $G_{i-1}$ by a unification operation. 
Note that every graph in a unification sequence that begins with a quartet graph is also a quartet graph.
A unification sequence is \emph{complete} if $G_k$ has no edges. See Fig.~\ref{figQuartetGraph}(a)-(d) for an example of a complete unification sequence. 
The following theorem is from \cite{Grunewald2008a}.

\begin{theorem}\label{thm:quartet_graphs}
A collection $Q$ of quartets is compatible if and only if there exists a complete unification sequence for the quartet graph $G_{Q}$.
\end{theorem}

For a quartet graph $G$, 
we define the {\em quartet set corresponding to} $G$ as the set of quartets
$$
Q_{G} = \{ ab|cd : \mbox{there exists edges $ab$ and $cd$ of the same color in $G$}\}.
$$

\section{Compatibility of Quartets}


\begin{figure}[!b]
\centering
\subfloat[The quartet graph for $Q_{4,3}$.]{\label{figQ43}
\begin{tikzpicture}[scale=1.6]
\node (a1) at (1,.5) [circle,inner sep=1.5,fill=black,label=above:$a_1$] {};
\node (a2) at (2,.5) [circle,inner sep=1.5,fill=black,label=above:$a_2$] {};
\node (a3) at (3,.5) [circle,inner sep=1.5,fill=black,label=above:$a_3$] {};
\node (a4) at (4,.5) [circle,inner sep=1.5,fill=black,label=above:$a_4$] {};
\node (b1) at (1.5,-.5) [circle,inner sep=1.5,fill=black,label=below:$b_1$] {};
\node (b2) at (2.5,-.5) [circle,inner sep=1.5,fill=black,label=below:$b_2$] {};
\node (b3) at (3.5,-.5) [circle,inner sep=1.5,fill=black,label=below:$b_3$] {};

\path[-] (a1) edge[out=40, in=140] node[above, inner sep=1]{$q_{1,1}$} (a2);
\path[-] (a1) edge[out=320, in=220] node[above, inner sep=1]{$q_{1,2}$}(a2);
\path[-] (a2) edge[out=40, in=140] node[above, inner sep=1]{$q_{2,1}$} (a3);
\path[-] (a2) edge[out=320, in=220] node[above, inner sep=1]{$q_{2,2}$} (a3);
\path[-] (a3) edge[out=40, in=140] node[above, inner sep=1]{$q_{3,1}$} (a4);
\path[-] (a3) edge[out=320, in=220] node[above, inner sep=1]{$q_{3,2}$} (a4);

\path[-] (b1) edge[out=40, in=140] node[above, inner sep=1]{$q_{1,1}$} (b2);
\path[-] (b1) edge[out=0, in=180] node[above, inner sep=1]{$q_{2,1}$}(b2);
\path[-] (b2) edge[out=40, in=140] node[above, inner sep=1]{$q_{1,2}$} (b3);
\path[-] (b2) edge[out=0, in=180] node[above, inner sep=1]{$q_{2,2}$} (b3);
\path[-] (b1) edge[out=320, in=220] node[above, inner sep=1]{$q_{3,1}$} (b2);
\path[-] (b2) edge[out=320, in=220] node[above, inner sep=1]{$q_{3,2}$}(b3);

\path[-] (a1) edge node[below left, inner sep=.5]{$q_{0}$} (b1);
\path[-] (a4) edge node[below right, inner sep=1]{$q_{0}$} (b3);

\end{tikzpicture}
}
\quad
\subfloat[$\{a_2,a_4\} \rightarrow u$]{\label{figQ43Unification}
\hspace{24pt}
\begin{tikzpicture}[scale=1.6]
\node (a1) at (1,.5) [circle,inner sep=1.5,fill=black,label=above:$a_1$] {};
\node (a24) at (2,.5) [circle,inner sep=1.5,fill=black,label=above:$u$\,\,] {};
\node (a3) at (3,.5) [circle,inner sep=1.5,fill=black,label=right:$a_3$] {};
\node (b1) at (1,-.5) [circle,inner sep=1.5,fill=black,label=below:$b_1$] {};
\node (b2) at (2,-.5) [circle,inner sep=1.5,fill=black,label=below:$b_2$] {};
\node (b3) at (3,-.5) [circle,inner sep=1.5,fill=black,label=below:$b_3$] {};

\path[-] (a1) edge[out=40, in=140] node[above, inner sep=1]{$q_{1,1}$} (a24);
\path[-] (a1) edge[out=320, in=220] node[above, inner sep=1]{$q_{1,2}$}(a24);

\path[-] (a24) edge[out=70, in=110] node[above, inner sep=1]{$q_{2,1}$} (a3);
\path[-] (a24) edge[out=25, in=155] node[above, inner sep=1]{$q_{2,2}$} (a3);
\path[-] (a24) edge[out=350, in=190] node[above, inner sep=1]{$q_{3,1}$} (a3);
\path[-] (a24) edge[out=310, in=230] node[above, inner sep=1]{$q_{3,2}$} (a3);

\path[-] (b1) edge[out=40, in=140] node[above, inner sep=1]{$q_{1,1}$} (b2);
\path[-] (b1) edge[out=0, in=180] node[above, inner sep=1]{$q_{2,1}$}(b2);
\path[-] (b2) edge[out=40, in=140] node[above, inner sep=1]{$q_{1,2}$} (b3);
\path[-] (b2) edge[out=0, in=180] node[above, inner sep=1]{$q_{2,2}$} (b3);
\path[-] (b1) edge[out=320, in=220] node[above, inner sep=1]{$q_{3,1}$} (b2);
\path[-] (b2) edge[out=320, in=220] node[above, inner sep=1]{$q_{3,2}$}(b3);

\path[-] (a1) edge node[below left, inner sep=.5]{$q_{0}$} (b1);
\path[-] (a24) edge node[above right, inner sep=1]{$q_{0}$} (b3);

\end{tikzpicture}
\hspace{24pt}
}
\caption{The set of quartets $Q_{4,3}$.}
\label{figQ43QuartetGraph}
\end{figure}
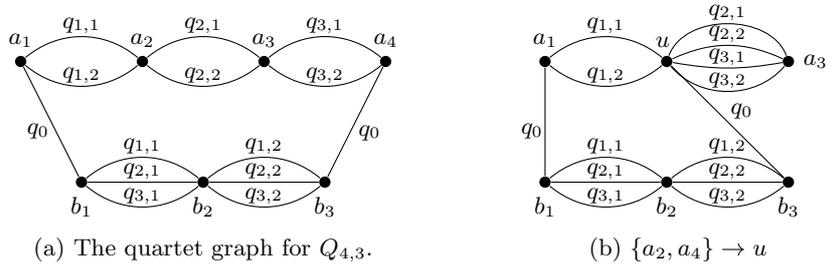

For every $s, t \ge 2$, we fix a set of labels $\L_{s,t} = \{a_1, a_2, \ldots, a_s, b_1, b_2, \ldots, b_t\}$ and define the set
\begin{equation*}
Q_{s,t} = \{ a_1 b_1 | a_s b_t \} \cup \bigcup_{i = 1}^{s-1} \bigcup_{j = 1}^{t-1} \{ a_i a_{i+1} | b_j b_{j+1} \}
\end{equation*}
of quartets with label set $\L_{s,t}$. 
We  denote the quartet $a_1 b_1 | a_s b_t$ by $q_0$, and a quartet of the form $a_i a_{i+1} | b_j b_{j+1}$ by $q_{i,j}$. 
See Fig.~\ref{figQ43QuartetGraph}(a) for an illustration. 

\begin{observation}\label{obsQSize}
For all $s,t \ge 2$, $|Q_{s,t}| = (s-1)(t-1)$ + 1.
\end{observation}

The proof of the next lemma is straightforward and given in the appendix.

\begin{lemma}\label{lm:st_equality}
$Q_{s, t}$ is compatible if and only if $Q_{t, s}$ is compatible.
\end{lemma}

\begin{lemma}\label{lemQIncompatible}
For all $s,t \ge 2$, $Q_{s,t}$ is incompatible.
\end{lemma}

\begin{proof}
We use induction on the size of $s + t$. 

\emph{Base case:} Let $s=2$, $t=2$ and thus, $s+t = 4$. The set $Q_{2, 2}$ contains only the two incompatible quartets $q_0=a_1b_1|a_2b_2$ and $q_{1,1}a_1a_2|b_1b_2$. Thus, $Q_{2,2}$ is incompatible. 

\emph{Induction step:} 
Assume that for every $s' <= s$ and $t' <= t$, where $s \geq 2$, $t \geq 2$, $s' \geq 2$, $t' \geq 2$ and $s + t \ge s' + t' + 1$, we have that $Q_{s',t'}$ is incompatible. 
To prove that $Q_{s,t}$ is incompatible, we will show that there is no complete unification sequence for the quartet graph $G=G_{Q_{s,t}}$. 
For sake of contradiction, assume that 
there exists a complete unification sequence $\S$ for $G$. 
Let $U$ be the set of vertices unified in the first unification operation of $\S$. 

For each $i \in [s]$ and $j \in [t]$, there exists edges of the same color incident on $a_i$ and $b_j$. Thus, if $U$ contains a vertex $a_i$, it cannot contain a vertex $b_j$. 
Then, w.l.o.g., by Lemma \ref{lm:st_equality}, we may assume that $U$ does not contain any vertex $b_j$ for $j \in [t]$.
Also, since the quartet $q_0=a_1b_1| a_sb_{t} \in Q_{s,t}$, it cannot be the case that $U$ contains both $a_1$ and $a_s$.  W.l.o.g., we assume that $U$ does not contain $a_1$. 

Let $a_x$ and $a_y$ be the two vertices in $U$ where $x$ is the smallest index over all of the vertices in $U$, and $y$ is the largest index over all of the vertices in $U$. Let $G'$ be the graph resulting from the unification of the vertices of $U$ in $G$, and let $u$ be the unique vertex in $V(G') \setminus V(G)$. Note that all edges between $a_{x-1}$ and $a_{x}$ in $G$ become edges between $a_{x-1}$ and $u$ in $G'$, and all edges between $a_{y}$ and $a_{y+1}$ in $G$ become edges between $u$ and $a_{y+1}$ in $G'$. See Fig~\ref{figQ43QuartetGraph} for an illustration. Let $Q=Q_{G'}$. 
Since the unification sequence $\S = G,G',G_1,\cdots,G_k$ is complete, the sequence $G',G_1,\cdots,G_k$ is also complete. Hence, the quartet set $Q$ is compatible.

Let $L = L(Q_{G'})$. 
Consider the quartet set $Q_{s-(y-x),t}$. 
Since $|U| > 1$, and, by assumption, $a_1 \not\in U$, we have that $1 < x < y \le s$.
Also, by assumption, $t \geq 2$. Hence, $Q_{s-y+x, t}$ is well defined. 
Let $h$ be an injective mapping from $L_{s-y+x, t}$ to $L$ defined by 
$$
	h(\ell) = \begin{cases}
		a_i & \mbox{ if } \ell \mbox{ is of the form } a_i \mbox{ where } 1 \le i < x\\
		u &\mbox{ if } \ell \mbox{ is of the form } a_i \mbox{ where } i = x\\
		a_{i+y-x} & \mbox{ if } \ell \mbox{ is of the form } a_i \mbox{ where } x < i \le s-y+x\\
		b_i & \mbox{ if } \ell \mbox{ is of the form } b_i \mbox{ where } 1 \le i \le t\\
	\end{cases}
$$
We will 
show that for every quartet 
$q=\ell_1 \ell_2 | \ell_3 \ell_4 \in Q_{s-y+x, t}$, there exists a quartet $h(\ell_1)h(\ell_2)|h(\ell_3)h(\ell_4) \in Q$.
Since $y>x$, $s-y+x+t < s+t$. 
By the inductive hypothesis, $Q_{s-y+x,t}$ is incompatible. It follows that $Q$ contains an incompatible subset of quartets, contradicting that $Q$ is compatible.
We have the following cases.
\begin{itemize}
\item[] {Case 1:} 
	$q=q_0$. 
	Then $\ell_1=a_1$, $\ell_2=b_1$, $\ell_3=a_{s-y+x}$, and $\ell_4=b_t$. 
	So, $h(\ell_1) = a_1$, $h(\ell_2)=b_1$, and $h(\ell_4)=b_t$.
	If $y=s$, then $h(\ell_3) = h(a_{s-y+x}) = h(a_x) = u$, and $a_1 b_1 | u b_t \in Q$. 
	If $y<s$, then $h(\ell_3) = h(a_{s-y+x}) = a_{s-y+x+y-x} = a_s$, and $a_1b_1|a_sb_t \in Q$.\\
\item[] {Case 2:} 
	$q=q_{i,j}$ for some $1 \le i < s-y+x$ and $1 \le j < t$. 
	Then $\ell_1=a_i$ and $\ell_2=a_{i+1}$, $\ell_3=b_j$, and $\ell_4=b_{j+1}$. 
	So, $h(\ell_3)=b_j$, and $h(\ell_4)=b_{j+1}$.
	Note that since $i < s - y + x$, we have that if $i \ge x$, then $y < s$.
	Hence, we have the following four possibilities to consider.\\
\begin{itemize}
\item[] {Case 2a:} 
		$1 \le i < x-1$.
		Since both $i < x$ and $i+1 < x$, $h(\ell_1) = a_i$, $h(\ell_2)=a_{i+1}$,
		and $a_ia_{i+1}|b_jb_{j+1} \in Q$.\\
\item[] {Case 2b:} 
	$i=x-1$.
	Then $h(\ell_1)\!=\!a_{x-1}$, $h(\ell_2)=u$, and $a_{x-1} u | b_j b_{j+1} \in Q$.\\
\item[] {Case 2c:} 
	$i=x$ and $y<s$.
	Then, $h(\ell_1)=h(a_x)=u$, $h(\ell_2)=h(a_{x+1})=a_{x+1+y-x}=a_{y+1}$,  and $ua_{y+1}|b_jb_{j+1} \in Q$.\\
\item[] {Case 2d:} 
	$i>x$ and $y<s$.
	Then, $h(\ell_1) = h(a_i) = a_{i+y-x}$ and $h(\ell_2)= h(a_{i+1}) = a_{i+1+y-x}$.
	Since $i>x$, it follows that both $ i + y - x > y$ and $ i + 1+ y - x > y$. 
	Since $i < s-y+x$, it follows that both $ i + y - x < s$ and $ i + y - x + 1 <= s$.
	Hence, $a_{i+y-x} a_{i+1+y-x}|b_jb_{j+1} \in Q$.
\end{itemize}
\end{itemize}
In every case, we have shown that $h(\ell_1)h(\ell_2)|h(\ell_3)h(\ell_4) \in Q$.
\qed
\end{proof}

\begin{lemma}\label{lemQSubsetCompatible}
For all $s,t \ge 2$, and every $q \in Q_{s,t}$, $Q_{s,t} \setminus \{q\}$ is compatible.
\end{lemma}

\begin{proof}
Let $q \in Q_{s,t}$. Either $q=q_0$ or $q=q_{x,y}$ for some $1 \le x < s$ and $1 \le y < t$. In every case, we present a tree witnessing that $Q_{s,t} \setminus \{q\}$ is compatible.
\begin{itemize}
\item[] {\em Case 1.} Suppose $q=q_0$. Create the tree $T$ as follows: There is a node for each label in $\L_{s,t}$, and two additional nodes $a$ and $b$. There is an edge $ab$. For every $a_x \in \L_{s,t}$ there is an edge $a_x a$. For every $b_x \in \L_{s,t}$, there is an edge $b_x b$. There are no other nodes or edges in $T$. See Fig. \ref{figMissingQuartet}(a) for an illustration of $T$. Consider any quartet $q \in Q_{s,t} \setminus \{q_0\}$. Then $q = a_i a_{i+1} | b_j b_{j+1}$ for some $1 \le i < s$ and $1 \le j < t$. Then, the minimal subgraph of $T$ connecting leaves with labels in $\{a_i, a_{i+1}, b_j, b_{j+1}\}$ is the quartet $q$.\\
\item[] {\em Case 2.} Suppose $q=q_{x,y}$ for some $1 \le x < s$ and $1 \le y < t$. Create the tree $T$ as follows: There is a node for each label in $\L_{s,t}$ and six additional nodes $\alow$, $\blow$, $\ell$, $h$, $\ahigh$, and $\bhigh$. There are edges $\alow \ell$, $\blow \ell$, $\ell h$, $h \ahigh$, and $h \bhigh$. For every $a_i \in \L_{s,t}$, there is an edge $a_i \alow$ if $i \le x$, and an edge $a_i \ahigh$ if $i > x$. For every $b_j \in \L_{s,t}$ there is an edge $b_j \blow$ if $j \le x$, and an edge $b_j \bhigh$ if $j > y$. There are no other nodes or edges in $T$. See Fig. \ref{figMissingQuartet}(b).  Now consider any quartet $q \in Q_{s,t} \setminus \{q_{x,y}\}$. Either $q=q_0$ or $q = q_{i,j}$ where $i \ne x$ or $j \ne y$. If $q=q_0$, then the minimal subgraph of $T$ connecting leaves with labels in $\{a_1, b_1, a_s, b_t\}$ is the subtree of $T$ induced by the nodes in $\{a_1, \alow, \ell, \blow, b_1, a_s, \ahigh, h, \bhigh, b_t\}$. Suppressing all degree two vertices results in a tree that is the same as $q_0$. So $T$ displays $q$. So assume that $q=a_i a_{i+1} | b_j b_{j+1}$ where $i \ne x$ or $j \ne y$. We define the following subset of the nodes in $T$:
\begin{equation*}
V = \begin{cases}
	\{a_i, a_{i+1}, \alow, \ell, \blow, b_j, b_{j+1} \} & \mbox{ if }  i < x \mbox{ and } j < y,\\
	\{a_i, a_{i+1}, \alow, \ell, b_y, \blow, h, \bhigh, b_{y+1} \} & \mbox{ if } i < x \mbox{ and } j = y,\\
	\{a_i, a_{i+1}, \alow, \ell, h, \bhigh, b_j, b_{j+1} \} & \mbox{ if } i < x \mbox{ and } j > y,\\
	\{a_x, \alow, \ell, h, \ahigh, a_{x+1}, \blow, b_j, b_{j+1} \} & \mbox{ if } i = x \mbox{ and } j < y,\\
	\{a_x, \alow, \ell, h, \ahigh, a_{x+1}, \bhigh, b_j, b_{j+1} \} & \mbox{ if } i = x \mbox{ and } j > y,\\
	\{a_j, a_{j+1},\ahigh, h, \ell, \blow, b_j, b_{j+1} \}  & \mbox{ if } i > x \mbox{ and } j < y,\\
	\{a_j, a_{j+1},\ahigh, h, b_y, \blow, \ell, \bhigh, b_{y+1} \} & \mbox{ if } i > x \mbox{ and } j = y,\\
	\{a_j, a_{j+1},\ahigh, h, \bhigh, b_j, b_{j+1} \} & \mbox{ if } i > x \mbox{ and } j > y.
\end{cases}
\end{equation*}
Now, the subgraph of $T$ induced by the nodes in $V$ is the minimal subgraph of $T$ connecting leaves with labels in $q$. Suppressing all degree two vertices results in a tree that is the same as $q$. Hence, $T$ displays $q$. \qed
\end{itemize}
\end{proof}


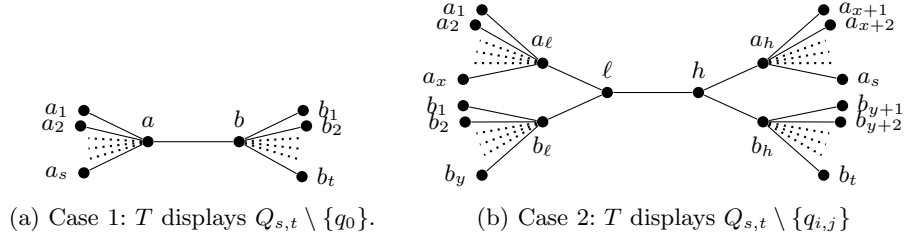
\begin{figure}[!t]
\centering
\subfloat[Case 1: $T$ displays $Q_{s,t} \setminus \{q_0\}$.]
{
	\centering
	\hspace{5pt}
	\begin{tikzpicture}[scale=1]
	\node (left) at (-.6,0) [circle,inner sep=1.5,fill=black,label=above:$a$] {};
	\node (right) at (.6,0) [circle,inner sep=1.5,fill=black,label=above:$b$] {};
	\node (a1) at (164:1.5) [circle,inner sep=1.5,fill=black,label=left:$a_1$] {};
	\node (a2) at (172:1.5) [circle,inner sep=1.5,fill=black,label=left:$a_2$] {};
	\node (left1) at (178:1.5) {};
	\node (left2) at (184:1.5) {};
	\node (left3) at (190:1.5) {};
	\node (as) at (196:1.5) [circle,inner sep=1.5,fill=black,label=left:$a_s$] {};
	\node (b1) at (16:1.5) [circle,inner sep=1.5,fill=black,label=right:$b_1$] {};
	\node (b2) at (8:1.5) [circle,inner sep=1.5,fill=black,label=right:$b_2$] {};
	\node (right1) at (2:1.5) {};
	\node (right2) at (356:1.5) {};
	\node (right3) at (350:1.5) {};
	\node (bt) at (342:1.5) [circle,inner sep=1.5,fill=black,label=right:$b_t$] {};
	\path[-] (left) edge (right);
	\path[-] (left) edge (a1);
	\path[-] (left) edge (a2);
	\path[dotted,thick] (left) edge (left1);
	\path[dotted,thick] (left) edge (left2);
	\path[dotted,thick] (left) edge (left3);
	\path[-] (left) edge (as);
	\path[-] (b1) edge (right);
	\path[-] (b2) edge (right);
	\path[-] (bt) edge (right);
	\path[dotted,thick] (right1) edge (right);
	\path[dotted,thick] (right2) edge (right);
	\path[dotted,thick] (right3) edge (right);
	\end{tikzpicture}
	\hspace{5pt}
}
\quad
\subfloat[Case 2: $T$ displays $Q_{s,t} \setminus \{q_{i,j}\}$]
{
	\centering
	\begin{tikzpicture}[scale=1]
	\node (low) at (-.6,0) [circle,inner sep=1.5,fill=black,label=above:$\ell$] {};
	\node (high) at (.6,0) [circle,inner sep=1.5,fill=black,label=above:$h$] {};
	\node (alow) at (165:1.5) [circle,inner sep=1.5,fill=black,label=above:$\alow$] {};
	\node (ahigh) at (15:1.5) [circle,inner sep=1.5,fill=black,label= above:$\ahigh$] {};
	\node (blow) at (195:1.5) [circle,inner sep=1.5,fill=black,label= below:$\blow$] {};
	\node (bhigh) at (345:1.5) [circle,inner sep=1.5,fill=black,label= below:$\bhigh$] {};
	\path[-] (low) edge (high);
	\path[-] (low) edge (alow);
	\path[-] (low) edge (blow);
	\path[-] (ahigh) edge (high);
	\path[-] (bhigh) edge (high);
	
	\node (a1) at (154:2.5) [circle,inner sep=1.5,fill=black,label=left:$a_1$] {};
	\node (a2) at (159:2.5) [circle,inner sep=1.5,fill=black,label=left:$a_2$] {};
	\node (alow1) at (163:2.5) {};
	\node (alow2) at (167:2.5) {};
	\node (alow3) at (171:2.5) {};
	\node (ax) at (176:2.5) [circle,inner sep=1.5,fill=black,label=left:$a_x$] {};
	\path[-] (alow) edge (a1);
	\path[-] (alow) edge (a2);
	\path[dotted,thick] (alow) edge (alow1);
	\path[dotted,thick] (alow) edge (alow2);
	\path[dotted,thick] (alow) edge (alow3);
	\path[-] (alow) edge (ax);
	
	\node (b1) at (184:2.5) [circle,inner sep=1.5,fill=black,label=left:$b_1$] {};
	\node (b2) at (189:2.5) [circle,inner sep=1.5,fill=black,label=left:$b_2$] {};
	\node (blow1) at (193:2.5) {};
	\node (blow2) at (197:2.5) {};
	\node (blow3) at (201:2.5) {};
	\node (by) at (206:2.5) [circle,inner sep=1.5,fill=black,label=left:$b_y$] {};
	\path[-] (blow) edge (b1);
	\path[-] (blow) edge (b2);
	\path[dotted,thick] (blow) edge (blow1);
	\path[dotted,thick] (blow) edge (blow2);
	\path[dotted,thick] (blow) edge (blow3);
	\path[-] (blow) edge (by);
	
	\node (ax1) at (26:2.5) [circle,inner sep=1.5,fill=black,label=right:$a_{x+1}$] {};
	\node (ax2) at (21:2.5) [circle,inner sep=1.5,fill=black,label=right:$a_{x+2}$] {};
	\node (ahigh1) at (17:2.5) {};
	\node (ahigh2) at (13:2.5) {};
	\node (ahigh3) at (9:2.5) {};
	\node (as) at (4:2.5) [circle,inner sep=1.5,fill=black,label=right:$a_s$] {};
	\path[-] (ahigh) edge (ax1);
	\path[-] (ahigh) edge (ax2);
	\path[dotted,thick] (ahigh) edge (ahigh1);
	\path[dotted,thick] (ahigh) edge (ahigh2);
	\path[dotted,thick] (ahigh) edge (ahigh3);
	\path[-] (ahigh) edge (as);
	
	\node (by1) at (356:2.5) [circle,inner sep=1.5,fill=black,label=right:$b_{y+1}$] {};
	\node (by2) at (351:2.5) [circle,inner sep=1.5,fill=black,label=right:$b_{y+2}$] {};
	\node (bhigh1) at (347:2.5) {};
	\node (bhigh2) at (343:2.5) {};
	\node (bhigh3) at (339:2.5) {};
	\node (bt) at (334:2.5) [circle,inner sep=1.5,fill=black,label=right:$b_t$] {};
	\path[-] (bhigh) edge (by1);
	\path[-] (bhigh) edge (by2);
	\path[dotted,thick] (bhigh) edge (bhigh1);
	\path[dotted,thick] (bhigh) edge (bhigh2);
	\path[dotted,thick] (bhigh) edge (bhigh3);
	\path[-] (bhigh) edge (bt);
	\end{tikzpicture}
}
\caption{Illustrating the proof of Lemma \ref{lemQSubsetCompatible}.
}
\label{figMissingQuartet}
\end{figure}

\begin{theorem}\label{thmQuartetLB}
For every integer $n \ge 4$, there exists a set $Q$ of quartets over $n$ taxa such that all of the following conditions hold.
\begin{enumerate}
	\item $Q$ is incompatible.
	\item Every proper subset of $Q$ is compatible.
	\item $|Q|=\lfloor\frac{n-2}{2}\rfloor\cdot\lceil\frac{n-2}{2}\rceil + 1$.
\end{enumerate}
\end{theorem}
\begin{proof}
By letting $s=\lfloor\frac{n}{2}\rfloor$ and $t=\lceil\frac{n}{2}\rceil$, it follows from Observation~\ref{obsQSize} and Lemmas~\ref{lemQIncompatible} and~\ref{lemQSubsetCompatible} that $Q_{s,t}$ is of size $(\lfloor\frac{n}{2}\rfloor - 1)\cdot(\lceil\frac{n}{2}\rceil - 1) + 1 = \lfloor\frac{n-2}{2}\rfloor\cdot\lceil\frac{n-2}{2}\rceil + 1$ and is incompatible, but every proper subset of $Q_{s,t}$ is compatible.
\qed
\end{proof}

The following theorem allows us to use our result on quartet compatibility to establish a lower bound on $f(r)$.
\begin{theorem}\label{thmLBByQuartets} Let $Q$ be a set of incompatible quartets over $n$ labels such that every proper subset of $Q$ is compatible, and let $r=n-2$. Then, there exists a set $C$ of $|Q|$ $r$-state characters such that $C$ is incompatible, but every proper subset of $C$ is compatible.
\end{theorem}
\begin{proof}
We claim that $C_Q$ is such a set of incompatible $r$-state characters. Since for two quartets $q_1,q_2 \in Q$, $\chi_{q_1}\ne\chi_{q_2}$, it follows that $|C_Q|=|Q|$. Since $Q$ is incompatible, it follows by Lemma \ref{lemQtoC} that $C_Q$ is incompatible. Let $C'$ be any proper subset of $C$. Then, there is a proper subset $Q'$ of $Q$ such that $C' = C_{Q'}$. Then, since $Q'$ is compatible, it follows by Lemma \ref{lemQtoC} that $C'$ is compatible.
\qed 
\end{proof}

Theorem \ref{thmQuartetLB} together with Theorem \ref{thmLBByQuartets} gives the following theorem.

\begin{theorem}\label{lemCharacterLB} 
For every integer $r \ge 2$, there exists a set $C$ of $r$-state characters such that all of the following hold.
\begin{enumerate}
	\item $C$ is incompatible.
	\item Every proper subset of $C$ is compatible.
	\item $|C| = \lfloor\frac{r}{2}\rfloor\cdot\lceil\frac{r}{2}\rceil + 1$.
\end{enumerate}
\end{theorem}

\begin{proof}
By Theorem \ref{thmQuartetLB} and Observation \ref{obsQSize}, there exists a set $Q$ of $\lfloor\frac{r}{2}\rfloor\cdot\lceil\frac{r}{2}\rceil + 1$ quartets over $r+2$ labels that that are incompatible, but every proper subset is compatible, namely $Q_{\lfloor\frac{r+2}{2}\rfloor,\lceil\frac{r+2}{2}\rceil}$. The theorem follows from Theorem \ref{thmLBByQuartets}.
\qed
\end{proof}

The quadratic lower bound on $f(r)$ follows from Theorem \ref{lemCharacterLB}.

\begin{corollary}\label{corFrLB}
$f(r) \ge \lfloor\frac{r}{2}\rfloor\cdot\lceil\frac{r}{2}\rceil + 1$.
\end{corollary}

\section{Compatibility of Triplets}


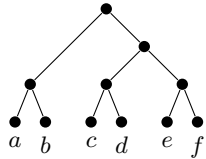
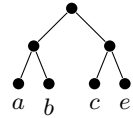
\begin{figure}[!b]
\centering
\subfloat[A tree $T$ witnessing that the triplets $ab|c$, $de|b$, $ef|c$, and $ec|b$ are compatible]{
\centering
\hspace{50pt}
\begin{tikzpicture}[scale=1]
\node(root) at (0,0) [circle,inner sep=1.5,fill=black] {};
\node (left) at (-1,-1) [circle,inner sep=1.5,fill=black] {};
\node (right) at (.5,-.5) [circle,inner sep=1.5,fill=black] {};
\node (v2) at (0,-1) [circle,inner sep=1.5,fill=black] {};
\node (v3) at (1,-1) [circle,inner sep=1.5,fill=black] {};
\node (a) at (-1.2,-1.5) [circle,inner sep=1.5,fill=black,label=below:$a$] {};
\node (b) at (-.8,-1.5) [circle,inner sep=1.5,fill=black,label=below:$b$] {};
\node (c) at (-.2,-1.5) [circle,inner sep=1.5,fill=black,label=below:$c$] {};
\node (d) at (.2,-1.5) [circle,inner sep=1.5,fill=black,label=below:$d$] {};
\node (e) at (.8,-1.5) [circle,inner sep=1.5,fill=black,label=below:$e$] {};
\node (f) at (1.2,-1.5) [circle,inner sep=1.5,fill=black,label=below:$f$] {};
\path[-] (root) edge (left);
\path[-] (root) edge (right);
\path[-] (right) edge (v2);
\path[-] (right) edge (v3);
\path[-] (left) edge (a);
\path[-] (left) edge (b);
\path[-] (v2) edge (c);
\path[-] (v2) edge (d);
\path[-] (v3) edge (e);
\path[-] (v3) edge (f);
\end{tikzpicture}
\hspace{50pt}
}
\quad
\subfloat[$T|\{a,b,c,e\}$]{
\centering
\hspace{30pt}
\begin{tikzpicture}[scale=1]
\node(root) at (0,0) [circle,inner sep=1.5,fill=black] {};
\node (left) at (-.5,-.5) [circle,inner sep=1.5,fill=black] {};
\node (right) at (.5,-.5) [circle,inner sep=1.5,fill=black] {};
\node (a) at (-.7,-1) [circle,inner sep=1.5,fill=black,label=below:$a$] {};
\node (b) at (-.3,-1) [circle,inner sep=1.5,fill=black,label=below:$b$] {};
\node (c) at (.3,-1) [circle,inner sep=1.5,fill=black,label=below:$c$] {};
\node (e) at (.7,-1) [circle,inner sep=1.5,fill=black,label=below:$e$] {};
\path[-] (root) edge (left);
\path[-] (root) edge (right);
\path[-] (left) edge (a);
\path[-] (left) edge (b);
\path[-] (right) edge (c);
\path[-] (right) edge (e);
\end{tikzpicture}
\hspace{30pt}
}
\caption{Rooted Phylogenetic Trees}
\label{figRootedTree}
\end{figure}

A {\em rooted phylogenetic tree} (or just {\em rooted tree}) is  a tree whose leaves are in one to one correspondence with a label set $\L(T)$, has a distinguished vertex called the {\em root}, and no vertex other than root has degree two. See Fig. \ref{figRootedTree}(a) for an example. A rooted tree is \emph{binary} if the root vertex has degree two, and every other internal (non-leaf) vertex has degree three.
A {\em triplet} is a rooted binary tree with exactly three leaves. A triplet with label set $\{a, b, c\}$ is denoted $ab|c$ if the  path between the leaves labeled $a$ and $b$ does include the root vertex. For a tree $T$, and a label set $L \subseteq L(T)$, let $T'$ be the minimal subtree of $T$ connecting all the leaves with labels in $L$. The \emph{restriction} of $T$ to $L$, denoted by $T|L$, is the rooted tree obtained from $T'$ by distinguishing the vertex closest to the root of $T$ as the root of $T'$, and suppressing every vertex other than the root having degree two. A rooted tree $T$ \emph{displays} another rooted tree $T'$ if $T'$ can be obtained from $T|\L(T')$ by contracting edges. A rooted tree $T$ displays a collection of rooted trees $\T$ if $T$ displays every tree in $\T$. If such a tree $T$ exists, then we say that $\T$ is compatible; otherwise, we say that $\T$ is incompatible. Given a collection of rooted trees $\T$, it can be determined in polynomial time if $\T$ is compatible~\cite{Aho1981a}.

The following theorems follow from the connection between collections of unrooted trees with at least one common label across all the trees, and collections of rooted trees \cite{Steel1992a}.

\begin{theorem}\label{thm:quartets_to_triplets}
Let $Q$ be a collection of quartets where every quartet in $Q$ shares a common label $\ell$. Let $R$ be the set of triplets such that there exists a triplet $ab|c$ in $R$ if and only if there exists a quartet $ab|c\ell$ in $Q$. Then, $Q$ is compatible if and only if $R$ is compatible.
\end{theorem}

Let $R$ be a collection of triplets. For a subset $S \subseteq L(R)$, we define the graph $[R, S]$ as the graph having a vertex 
for each label in $S$, and an edge $\{a, b\}$ 
if and only if  $ab|c \in R$ for some $c \in S$. The following theorem is from~\cite{BryantSteel95}.

\begin{theorem}\label{thm:onetree}
A collection $R$ of rooted triplets is compatible if and only if $[R, S]$ is not connected for every $S \subseteq L(R)$ with $|S| \geq 3$.
\end{theorem}

\begin{corollary}\label{cor:disconnected_aho}
Let $R$ be a set of rooted triplets such that $R$ is incompatible but every proper subset of $R$ is compatible. Then, $[R, L(R)]$ is connected.
\end{corollary}

We now contrast our result on quartet compatibility with a result on triplets. 
\vspace{-2.1em}\\
\begin{theorem}
For every $n \ge 3$, if $R$ is an incompatible set of triplets over $n$ labels, and $|R| > n-1$, then some proper subset of $R$ is incompatible. 
\end{theorem}
\begin{proof}
For sake of contradiction, let $R$ be a set of triplets such that $R$ is incompatible, every proper subset of $R$ is compatible, $|L(R)| = n$, and $|R| > n-1$. The graph $[R, L(R)]$ will contain $n$ vertices and at least $n$ edges. Since each triplet in $R$ is distinct, there will be a cycle $C$ of length at least three in $[R, L(R)]$. Since $R$ is incompatible but every proper subset of $R$ is compatible, by Corollary~\ref{cor:disconnected_aho}, $[R, L(R)]$ is connected. 

Consider any edge $e$ in the cycle $C$. Let $t$ be the triplet that contributed edge $e$ in $[R, L(R)]$. Let $R' = R \setminus t$. Since the graph $[R, L(R)] - e$ is connected, $[R', L(R')]$ is connected. By theorem~\ref{thm:onetree}, $R'$ is incompatible. But $R' \subset R$, contradicting that 
every proper subset of 
$R$ is compatible.
\qed
\end{proof}

To show the bound is tight,\,we first prove a more restricted form of Theorem\,\ref{thmQuartetLB}.

\begin{theorem}\label{thmQuartetSharedLB}
For every $n \ge 4$, there exists a set of quartets $Q$ with $|L(Q)|=n$, and a label $\ell \in L(Q)$, such that all of the following hold.
\begin{enumerate}
\item Every $q \in Q$ contains a leaf labeled by $\ell$.
\item $Q$ is incompatible.
\item Every proper subset of $Q$ is compatible.
\item $|Q| = n-2$.
\end{enumerate}
\end{theorem}
\begin{proof}
Consider the set of quartets $Q_{2, n-2}$. From Lemmas~\ref{lemQIncompatible} and~\ref{lemQSubsetCompatible}, $Q_{2, n-2}$ is incompatible but every proper subset of $Q_{2, n-2}$ is compatible. The set $Q_{2, n-2}$ contains exactly $n-2$ quartets. From the construction, there are two labels in $L$ which are present in all the quartets in $Q_{2, n-2}$. Set one of them to be $\ell$.
\qed
\end{proof}

The following is a consequence of Theorem \ref{thmQuartetSharedLB} and Theorem \ref{thm:quartets_to_triplets}.

\begin{corollary}\label{corTriplets}
For every $n \ge 3$, there exists a set $R$ of triplets with $L(R)=n$ such that all of the following hold.
\begin{enumerate}
\item $R$ is incompatible.
\item Every proper subset of $R$ is compatible.
\item $|R| = n-1$.
\end{enumerate}
\end{corollary}

The generalization of the Fitch-Meacham examples  given in~\cite{Lam2011a} can also be expressed in terms of triplets. For any $r \geq 2$, let $L = \{a, b_1, b_2, \cdots, b_r\}$. Let 

\begin{equation*}
R_r = ab_r|b_1 \cup \bigcup_{i = 1}^{r-1} ab_i|b_{i+1}
\end{equation*}

Let $Q = \{ab|c\ell: ab|c \in R_r\}$ for some label $\ell \notin L$. The set $C_{Q}$ of $r$-state characters corresponding to the quartet set $Q$ is exactly the set of characters built for $r$ in~\cite{Lam2011a}. In the partition intersection graph of $C_{Q}$, (following the terminology in~\cite{Lam2011a}) labels $\ell$ and $a$ correspond to the end cliques and the rest of the $r$ labels $\{b_1, b_2, \cdots, b_r\}$ correspond to the $r$ tower cliques. From Lemma~\ref{lemQtoC} and Theorem~\ref{thm:quartets_to_triplets}, $R_r$ is compatible if and only of $Q$ is compatible.

\section{Conclusion}

We have shown that for every $r \ge 2$, $f(r) \ge \lfloor\frac{r}{2}\rfloor\cdot\lceil\frac{r}{2}\rceil + 1$, by showing that for every $n \ge 4$, there exists an incompatible set $Q$ of $\lfloor\frac{n-2}{2}\rfloor\cdot\lceil\frac{n-2}{2}\rceil + 1$ quartets over a set of $n$ labels such that every proper subset of $Q$ is compatible. Previous results show that our lower bound on $f(r)$ is tight for $r=2$ and $r=3$ \cite{Buneman1971a,Estabrook1976a,Gusfield1991a,Lam2011a,Meacham1983a,Semple2003a}. We give the following conjecture.

\begin{conj}\label{conj}
For every $r \ge 2$, $f(r) = \lfloor\frac{r}{2}\rfloor\cdot\lceil\frac{r}{2}\rceil + 1$.
\end{conj}

Note that, due to Theorem \ref{thmLBByQuartets}, a proof of Conjecture \ref{conj} would also show that the number of incompatible quartets given in the statement of Theorem \ref{thmQuartetLB} is also as large as possible. Another direction for future work is to show an upper bound on the function $f(r)$, which would prove Conjecture \ref{conjFR}. For quartets, we have a trivial upper bound of ${n \choose 4}$ on the cardinality of a set of quartets over $n$ labels such that every proper subset is compatible. However, the question for multi-state characters remains open.
\section*{Acknowledgments}
We thank Sylvain Guillemot and Mike Steel for valuable comments.
{
\bibliographystyle{splncs03}
\bibliography{bibliography}
}
\newpage
\section*{Technical Appendix}

\subsection{Proof of Lemma \ref{lemQtoC}}

\begin{proof}
Let $T$ be any tree. It suffices to show that $T$ displays $q=ab|cd$ if and only if $\chi_q$ is convex on $T$.
($\Rightarrow$) Suppose $T$ displays $q$. It follows that the path from the leaf labeled $a$ to the leaf labeled $b$ does not intersect the path from the leaf labeled $c$ to the leaf labeled $d$. Every other state $s$ of $\chi_q$ has only one label $\ell_s$ with state $s$ for $\chi_q$; hence the minimal subtree connecting all leaves with labels having state $s$ for $\chi$ is a single leaf. It follows that $\chi_q$ is convex on $T$. 
($\Leftarrow$) Now suppose that $\chi_q$ is convex on $T$. Then the minimal subtree containing both the leaf labeled $a$ and the leaf labeled $b$ is pairwise disjoint from the minimal subtree containing both the leaf labeled $c$ and the leaf labeled $d$.  It follows that the path from the leaf labeled $a$ to the leaf labeled $b$ does not intersect the path from the leaf labeled $c$ to the leaf labeled $d$. Hence $T$ displays $q$.
\qed
\end{proof}

\subsection{Proof of Lemma \ref{lm:st_equality}}

\begin{proof}
We define the function $g : L_{s,t} \rightarrow L_{t,s}$ by 
$$
	g(\ell) = \begin{cases}
				b_i	& \mbox{ if } \ell \mbox{ is of the form } a_i \mbox{ for some } 1 \le i \le s \\
				a_i & \mbox{ if } \ell \mbox{ is of the form } b_i \mbox{ for some } 1 \le i \le t. \\
			\end{cases}
$$
It is straightforward to verify that $\ell_1 \ell_2 | \ell_3 \ell_4 \in Q_{s,t}$ if and only if $g(\ell_1) g(\ell_2) | g(\ell_3) g(\ell_4) \in Q_{t,s}$. The theorem follows.
\qed
\end{proof}

\end{document}